\documentclass[12pt,a4paper,reqno]{amsart}
\usepackage[mathcal]{eucal}
\usepackage{amssymb}
\usepackage[nice]{nicefrac}
\usepackage{hyperref}
\usepackage[normalem]{ulem}
\usepackage{xcolor}

\theoremstyle{plain}
\newtheorem{theorem}{Theorem}[section]
\newtheorem{lemma}[theorem]{Lemma}
\newtheorem{proposition}[theorem]{Proposition}

\theoremstyle{definition}
\newtheorem{definition}[theorem]{Definition}

\newtheorem{remark}[theorem]{Remark}

\numberwithin{equation}{section}

\makeatletter
\newcommand{\AlignFootnote}[1]{%
    \ifmeasuring@
    \else
        \footnote{#1}%
    \fi
}
\makeatother

\DeclareMathOperator{\Imm}{Im}

\title{Strong conciseness of coprime commutators in profinite groups}
\author[I. de las Heras]{Iker de las Heras} 
\address{Iker de las Heras: Mathematisches Institut, Heinrich-Heine-Universit\"at, 40225
  D\"usseldorf, Germany; Department of Mathematics, Euskal Herriko Unibertsitatea UPV/EHU, 48940 Leioa, Spain}
\email{iker.delasheras@hhu.de; iker.delasheras@ehu.eus}

\author[M. Pintonello]{Matteo Pintonello} 
\address{Matteo Pintonello: Department of Mathematics, Euskal Herriko Unibertsitatea UPV/EHU, 48940 Leioa, Spain}
\email{matteo.pintonello@ehu.eus}

\author[P. Shumyatsky]{Pavel Shumyatsky} 
\address{Pavel Shumyatsky: Department of Mathematics, University of Brasilia, Brasilia DF, Brazil}
\email{pavel@unb.br}

\date{\today}

\thanks{
The first and second authors are supported by the Spanish Government project PID2020-117281GB-I00,  
partially by FEDER funds and by the Basque Government project IT483-22.
The first author has also received
funding from the European Union’s Horizon 2021 research and innovation programme under the
Marie Sklodowska-Curie grant agreement, project 101067088.
The second author is also supported by a grant FPI-2018 of the Spanish Government.
The third author was supported by FAPDF and CNPq.
}



\begin{document}

\maketitle

\begin{abstract}
Let $G$ be a profinite group. The coprime commutators $\gamma_j^*$ and $\delta_j^*$ are defined as follows. Every element of $G$ is  both a $\gamma_1^*$-value and a $\delta_0^*$-value. For $j\geq 2$, let $X$ be the set of all elements of $G$ that are powers of $\gamma_{j-1}^*$-values. An element $a$ is a $\gamma_j^*$-value if there exist $x\in X$ and $g\in G$ such that $a=[x,g]$ and $(|x|,|g|)=1$. For $j\geq 1$, let $Y$ be the set of all elements of $G$ that are powers of $\delta_{j-1}^*$-values. The element $a$ is a $\delta_j^*$-value if there exist $x,y\in Y$ such that $a=[x,y]$ and $(|x|,|y|)=1$. 

In this paper we establish the following results.

A profinite group $G$ is finite-by-pronilpotent if and only if there is $k$ such that the set of $\gamma_k^*$-values in $G$ has cardinality less than $2^{\aleph_0}$ (Theorem \ref{thm: gamma strongly concise}).

A profinite group $G$ is finite-by-(prosoluble of Fitting height at most $k$) if and only if there is $k$ such that the set of $\delta_k^*$-values in $G$ has cardinality less than $2^{\aleph_0}$ (Theorem \ref{thm: delta strongly concise}).

\end{abstract}

\section{Introduction}

A group word $w$ is called concise in the class of groups $\mathcal C$ if the verbal subgroup $w(G)$ is finite whenever $w$ takes only finitely many values in a group $G\in\mathcal C$. In the sixties Hall raised the problem whether every word is concise in the class of all groups but in 1989 S. Ivanov \cite{Iv89} solved the problem in the negative  (see also \cite[p.\ 439]{ols}). On the other hand, the problem for residually finite groups remains open (cf. Segal \cite[p.\ 15]{Segal} or Jaikin-Zapirain \cite{jaikin}). In recent years several new positive results with respect to this problem were obtained (see \cite{AS1, gushu,  fealcobershu, dms-conciseness, dms-bounded, DMS-engel-II,dmsisra}).

A natural variation of the notion of conciseness for profinite groups was introduced in \cite{dks20}: the word $w$ is strongly concise in a class of profinite groups $\mathcal{C}$ if the verbal subgroup $w(G)$ is finite in any group $G \in \mathcal{C}$ in which $w$ takes less than $2^{\aleph_0}$ values. Here and throughout the paper, whenever $G$ is a profinite group we write $w(G)$ to denote the {\it closed} subgroup generated by $w$-values. A number of new results on strong conciseness of group words can be found in \cite{dks20,detomi,azeshu,khushu}. 

It was noted in \cite{dms21} that the concept of (strong) conciseness can be applied in a wider context. Suppose $ \mathcal{C}$ is a class of profinite groups and $\phi(G)$ is a subset of $G$ for every $G \in\mathcal{C}$. Is the subgroup generated by $\phi(G)$ finite whenever $|\phi(G)| < 2^{\aleph_0}$? This question is interesting whenever $\phi(G)$ is defined in some natural way and/or properties of the subgroup $\langle\phi(G)\rangle$ have strong impact on the structure of $G$. It was shown in \cite{dms21} that the set of coprime commutators in a profinite group is strongly concise.

Given a profinite group $G$ and an element $x\in G$, we denote by $|x|$ (respectively $|G|$) the order of $x$ (respectively $G$) as a supernatural number (or Stenitz number), and $\pi(x)$ (respectively $\pi(G)$) will stand for the set of prime numbers dividing $|x|$ (respectively $|G|$). An element $a\in G$ is a coprime commutator if there are $x,y\in G$ such that $a=[x,y]$ and $(|x|,|y|)=1$. Here, as usual, $[x,y]$ stands for $x^{-1}y^{-1}x,y$. It is well-known that the set of coprime commutators in a profinite group $G$ generates the nilpotent residual $\gamma_\infty(G)$, that is, the smallest normal subgroup $N$ such that $G/N$ is pronilpotent. Of course, $\gamma_\infty(G)=\cap_i\gamma_i(G)$ is the intersection of the lower central series of $G$.

Coprime commutators of higher order in finite groups were defined in \cite{Sh15}. The definition naturally extends to the profinite case. Let $G$ be a profinite group. Every element of $G$ is  both a $\gamma_1^*$-value and a $\delta_0^*$-value. For $j\geq 2$, let $X$ be the set of all elements of $G$ that are powers (in the profinite sense) of $\gamma_{j-1}^*$-values. An element $a$ is a $\gamma_j^*$-value if there exist $x\in X$ and $g\in G$ such that $a=[x,g]$ and $(|x|,|g|)=1$. For $j\geq 1$, let $Y$ be the set of all elements of $G$ that are powers (in the profinite sense) of $\delta_{j-1}^*$-values. The element $a$ is a $\delta_j^*$-value if there exist $x,y\in Y$ such that $a=[x,y]$ and $(|x|,|y|)=1$. 

Set $D_1(G)=\gamma_\infty(G)$ and $D_i(G)=\gamma_\infty(D_{i-1})$ for $i\ge 2$. It was shown in \cite{Sh15} that for any $j$ we have the equalities $\gamma_j^*(G)=D_1(G)$ and $\delta_j^*(G)=D_j(G)$. Thus, $\gamma_j^*(G)=1$ if and only if $G$ is pronilpotent while $\delta_j^*(G)=1$ if and only if $G$ is prosoluble and has Fitting height at most $j$ (we say that a profinite group has Fitting height at most $h$ if it has a normal series of length $h$ all of whose factors are pronilpotent).

The main results in this paper are as follows.
\begin{theorem}
    \label{thm: gamma strongly concise}
 A profinite group $G$ is finite-by-pronilpotent if and only if there is $k$ such that the set of $\gamma_k^*$-values in $G$ has cardinality less than $2^{\aleph_0}$.\end{theorem}

\begin{theorem}
    \label{thm: delta strongly concise}
 A profinite group $G$ is finite-by-(prosoluble of Fitting height at most $k$) if and only if there is $k$ such that the set of $\delta_k^*$-values in $G$ has cardinality less than $2^{\aleph_0}$.
\end{theorem}

In the case where the set of $\gamma_k^*$-values (respectively, $\delta_j^*$-values) in a profinite group is finite of size $m$ the subgroup generated by the set is finite of order bounded in terms of $m$ only. This was established for finite groups in \cite{ast14} and can be extended to profinite groups in a straightforward manner.

\section{Preliminaries}

 For a subgroup $H$ of a profinite group $G$, we denote by $H^G$ the normal closure of $H$ in $G$, that is, the minimal normal subgroup of $G$ containing $H$. For an element $g\in G$ we denote by $g^G$ the conjugacy class of $g$ in $G$. We use the left-normed notation for commutators; for example, $[x,y,z] = [[x,y],z]$. If $X$ and $Y$ are subsets of $G$, then $[X,Y]$ stands for the subgroup generated by commutators $[x,y]$, where $x\in X$ and $y\in Y$. If $X=\{x\}$, then we just write $[x,Y]$.
If $U$ is an open normal subgroup of $G$, we write $U\trianglelefteq_o G$. All subgroups of profinite groups are assumed to be closed. 

One can easily see that if $N$ is a normal subgroup of $G$ and $x$ an element whose image in $G/N$ is a $\gamma_j^*$-value (respectively, a $\delta_j^*$-value), then there exists a $\gamma_j^*$-value $y$ in $G$ (respectively, a $\delta_j^*$-value) such that $x\in yN$. This observation will be used throughout the paper without explicit references.

We begin by quoting a theorem that was already mentioned in the introduction. It was established in \cite{Sh15} for finite groups. The variant for profinite groups is an obvious modification of the finite case.

\begin{theorem}[\cite{Sh15} Theorems 2.1, 2.7]\label{aaaa} Let $G$ be a profinite group. The subgroup $\gamma_k^*(G)$ is trivial if and only if $G$ is pronilpotent. The subgroup $\delta_k^*(G)$ is trivial if and only if $G$ is prosoluble of Fitting height at most $k$.
\end{theorem} 

The following results will be helpful. 

\begin{proposition}[\cite{dks20} Lemma 2.1]\label{prop: Klopsch}
    Let $\varphi:X\to Y$ be a continuous map between two non-empty profinite spaces that is nowhere locally constant (i.e. there is no non-empty open subset $U\subseteq_o X$ where $\varphi|_U$ is constant). Then $|\varphi(X)|\geq 2^{\aleph_0}$.
\end{proposition}

\begin{lemma}[\cite{dks20} Lemma 2.2]\label{lem: finite conjugacy class}
    Let $G$ be a profinite group and $g\in G$ be an element whose conjugacy class $g^G$ contains less than $2^{\aleph_0}$ elements. Then $g^G$ is finite.
\end{lemma}

It is well-known that if $A$ is a group of automorphisms of a finite group $G$ with $(|A|,|G|)=1$, then $[G,A]=[G,A,A]$.
The following lemma is a stronger version of this result for the case where $G$ is a pronilpotent group.

\begin{lemma}[\cite{KhSh21} Lemma 4.6]\label{lem: khukhro}
    Let $\varphi$ be a coprime automorphism of a pronilpotent group $G$. Then the restriction of the mapping
    $$\theta:x\to [x,\varphi]$$
    to the set $S=\{[g,\varphi]|g\in G\}$ is bijective.
\end{lemma}

The original version of Lemma \ref{lem: khukhro} states that the continuous map $\theta$ is injective when restricted to the set $S$.
However, since $S$ is closed and $\theta(S)\subseteq S$, the surjectivity of $\theta$ follows immediately.

The following is a profinite version of Lemma 2.4 in \cite{Sh15}.
It will be used together with Lemma \ref{lem: khukhro} in order to show that certain elements of a group $G$ are $\delta_k^*$-values.

\begin{lemma}\label{lem: long commutator of y_i are derived values}
    Let $G$ be a profinite group and $g_1,\ldots,g_k$ be $\delta^*_{k-1}$-values in $G$.
    Suppose that $g_1,\ldots, g_k\in N_G(H)$ for a subgroup $H\leq G$ with $(|H|,|g_i|)=1$ for every $i\in \{1,\ldots,k\}$.
    Then, for every $h\in H$, the element $[h,g_1,\ldots,g_k]$ is a $\delta_{k}^*$-value.
\end{lemma}

The next result is a profinite version of Lemma 2.4 in \cite{ast14}. By a meta-pronilpotent group we mean a profinite group $G$ having a normal pronilpotent subgroup $N$ such that $G/N$ is pronilpotent.

\begin{lemma}\label{lem: Hall acts on Sylow}
    Let $G$ be a meta-pronilpotent group. Then $\gamma_\infty(G)=\prod_p [K_p,H_{p'}]$, where $K_p$ is a Sylow $p$-subgroup of $\gamma_\infty (G)$ and $H_{p'}$ is a Hall $p'$-subgroup of $G$.
\end{lemma}

For a general group word $w$, the set $G_w$ of $w$-values of a profinite group $G$ is always closed in $G$.
We will show that the same is true for the sets of $\gamma_k^*$ and $\delta_k^*$-values. 

\begin{proposition}\label{lem: everything is closed}
    Let $S_1,\ldots,S_t$ be closed subsets of a profinite group $G$.
    Then the set
    $$C=\{(g_1,\ldots, g_t)\in S_1\times \cdots \times S_t \mid (|g_i|,|g_{i+1}|)=1 \text{ for all }i=1,\ldots, t-1\}$$
    is closed in $S_1\times \cdots \times S_t$. Furthermore, the sets $G_{\gamma_k^*}$ and $G_{\delta_k^*}$ are closed in $G$.
\end{proposition}
\begin{proof}
    Let $\mathcal{P}$ be the set of all primes and $p\in \mathcal{P}$.
    First notice that for every closed subset $S$ of $G$ the set
    $$S_{p'}=\{g\in S \mid p\notin \pi(g)\}$$
    is closed since $S_{p'}=\cap_{N\trianglelefteq_o G} S_{p'}N$.
    Also, the set $S^{\widehat{\mathbb{Z}}}=\{g^\lambda \mid g\in S, \lambda \in \widehat{\mathbb{Z}}\}$ is the image under the continuous map $f(g,\lambda)=g^\lambda$ of the compact set $S\times\widehat{\mathbb{Z}}$, so it is closed too.
    
    Let now $A,B$ be subsets of $G$. We claim that the set
    \begin{equation}\label{eq: closed set S}
    R_{A,B}=\bigcap_{p\in \mathcal{P}}\big((A\times  B_{p'}) \cup (A_{p'}\times B)\big)    
    \end{equation}
    is exactly the set of elements $(a,b)\in A\times B$ with $|a|$ and $|b|$ coprime.
    On the one hand, if $|a|$ and $|b|$ are coprime then $(a,b)\in (A\times  B_{p'}) \cup (A_{p'}\times B)$ for every $p\in \mathcal{P}$, because, if $b\in B\smallsetminus B_{p'}$, then $a\in A_{p'}$ necessarily.
    On the other hand, if $(a,b)\in R_{A,B}$ and a prime $p$ divides $|a|$, then $(a,b)\in (A, B_{p'})$ so $p$ does not divide $|b|$, and the claim follows.
    Notice now that if $A$ and $B$ are closed, the set $R_{A,B}$ is an intersection of closed subsets of $G\times G$ so it is closed too.
    
    It is now easy to prove by induction on $k$ that the sets $G_{\gamma_k^*}$, $G_{\delta_k^*}$ are closed: just note that $G_{\gamma_k^*}$ is exactly the set $R_{A,B}$ in (\ref{eq: closed set S}) with $A=G_{\gamma_{k-1}^*}^{\widehat{\mathbb{Z}}}$, $B=G$, whereas $G_{\delta_k^*}$ is the set $R_{A,B}$ in (\ref{eq: closed set S}) with $A=B=G_{\delta_{k-1}^*}^{\widehat{\mathbb{Z}}}$.
    
    To prove that the set $C$ is closed in $S_1\times\cdots \times S_t$, it suffices to notice that by the above arguments the set
    $$C_i=S_1\times \cdots \times S_{i-1}\times R_{S_i,S_{i+1}} \times S_{i+2}\times \cdots \times S_t$$
    is closed for every $i\in \{1,\ldots,t-1\}$ and $C=\bigcap_{i=1}^{t-1} C_i$.
\end{proof}

\section{Coprime commutators in prosoluble groups}
In this section we will prove some particular cases of Theorems \ref{thm: gamma strongly concise} and \ref{thm: delta strongly concise}.

First we will prove Theorem \ref{thm: gamma strongly concise} in the case when the profinite group $G$ is meta-pronilpotent.
The second part of the section will be devoted to the proof of Theorem \ref{thm: delta strongly concise} in the case where $G$ is poly-pronilpotent.

We first start with a general reduction.

\begin{lemma}\label{lem: reduce to abelian case}
    Let $\phi$ be a map that associates to every group $G$ a normal subset $\phi(G)\subseteq G$.
    Let $G$ be a profinite group with $|\phi(G)|< 2^{\aleph_0}$ and let $K$ be a pronilpotent subgroup of $\langle \phi(G)\rangle$ generated by a subset of $\phi(G)$.
    If $K/K'$ is finite, then $K$ is finite.
\end{lemma}
\begin{proof}
    Since $K$ is pronilpotent, $K' \leq \Phi(K)$, where $\Phi(K)$ stands for the Frattini subgroup of $K$.
    Thus $K/\Phi(K)$ is finite, and hence we can find a finite subset $S$ of $\phi(G)$ generating $K$.
    Since $\phi(G)$ is a normal subset of $G$, by Lemma \ref{lem: finite conjugacy class} each of these generators has finitely many conjugates in $G$, so in particular $|G:C_G(s)|\leq \infty$ for every $s\in S$. Since $C_G(K)=\cap_{s\in S} C_G(s)$,
   this implies that $Z(K)= K\cap C_G(K)$ has finite index in $K$, and by Schur's theorem $K'$ is finite.
\end{proof}

We will use Lemma \ref{lem: reduce to abelian case} for $\phi=\gamma_k^*$ or $\phi=\delta_k^*$, but it could be applied to other cases, such as group word maps or uniform (anti-coprime) commutators (see \cite{dms21} or  \cite{dms22}).

\subsection{The meta-pronilpotent case for \texorpdfstring{$\gamma_k^*$}{Y*k}}

The next lemma is given without a proof since this is an obvious modification of Lemma 3.1 in \cite{dms21}.
\begin{lemma}\label{lem: 3.1 of DeMoSh}
    Let $G$ be a profinite group that is a product of a subgroup $H$ and a normal pronilpotent subgroup $Q$ with $(|H|,|Q|)=1$.
    Suppose that
    $$
    |\{[h,q]\mid h\in H, q\in Q\}|<2^{\aleph_0}.
    $$
    Then $[H,Q]$ is finite.
\end{lemma}

Similarly, by using Lemma \ref{lem: 3.1 of DeMoSh} in place of Lemma 3.1 of \cite{dms21}, we can refine the proof of Lemma 3.4 of \cite{dms21} and obtain the following version.

\begin{lemma}\label{lem: meta-pronilpotent for gamma2}
    Let $G$ be a meta-pronilpotent group with
    $$
    |\{[g,h]\mid g\in G, h\in\gamma_{\infty}(G),(|g|,|h|)=1\}|<2^{\aleph_0}.
    $$
    Then $\gamma_{\infty}(G)$ is finite.
\end{lemma}

We are now ready to prove the strong conciseness of $\gamma_k^*$ in meta-pronilpotent groups.

\begin{proposition}\label{prop: meta-pronilpotent for gamma}
    Let $G$ be a meta-pronilpotent group with $|G_{\gamma_k^*}|<2^{\aleph_0}$. Then $\gamma_{\infty}(G)$ is finite.
\end{proposition}
\begin{proof}
    Let $g\in G$ and $h\in\gamma_{\infty}(G)$ such that $(|g|,|h|)=1$, and let $H$ be the minimal Hall subgroup of $\gamma_{\infty}(G)$ containing $h$.
    Since $H$ is pronilpotent, Lemma \ref{lem: khukhro} shows that there exists $h'\in H$ such that $[h,g]=[h',g,\overset{k-1}{\ldots},g]$, and therefore $[h,g]\in G_{\gamma_k^*(G)}$.
    Hence, we have
    $$
    |\{[g,h]\mid g\in G, h\in\gamma_{\infty}(G),(|g|,|h|)=1\}|<2^{\aleph_0},
    $$
    and we conclude by Lemma \ref{lem: meta-pronilpotent for gamma2}.
\end{proof}

\subsection{The poly-pronilpotent case for \texorpdfstring{$\delta_k^*$}{d*k}}

Recall that a Sylow basis of a profinite group $G$ is a family $\{P_i\}$ of Sylow subgroups of $G$, one for each prime in $\pi(G)$, such that $P_iP_j=P_jP_i$ for every $i,j$.
The normalizer of a Sylow basis, namely $T=\bigcap_i N_G(P_i)$, is pronilpotent and satisfies $G=T\gamma_\infty(G)$.
Basic results on Sylow bases can be found in Section 9.2 of \cite{Rob96}.
It is well-known that any prosoluble group admits a Sylow basis (see Proposition 2.3.9 of \cite{RiZa10}).
Moreover, if $G$ is meta-pronilpotent, it is easy to see that $\gamma_{\infty}(G)=[T,\gamma_{\infty}(G)]$.

For every profinite group $G$, the subgroup $\delta^*_k(G)$ coincides with the nilpotent residual $\gamma_\infty (\delta_{k-1}^*(G)) $ of $\delta_{k-1}^*(G)$ (see \cite{Sh15}).
Let $\{P_i\}$ be a Sylow basis of $G$ and observe that $\{P_i\cap\delta_j^*(G)\}$ is a Sylow basis of $\delta_j^*(G)$ for every $j\ge 0$.
Let $T_j$ be the normalizer in $\delta_j^*(G)$ of the Sylow basis $\{P_i\cap\delta_j^*(G)\}$, so that $G=T_0T_1\cdots T_k\delta^*_{k+1}(G)$ for every $k\ge 1$ and $T_j\le N_G(T_i)$ for every $j\le i$.
In particular, if $G$ is a prosoluble group of Fitting height $k$, then $\delta_{k+1}^*(G)=1$, and therefore $G=T_0\cdots T_k$.
In this case, we have $T_k=\delta_k^*(G)$ and $T_i=[T_0,\ldots, T_i] \pmod{\delta_{i+1}^*(G)}$ for all $i\in\{0,\ldots,k\}$.

For this reason, we will often work with subgroups $G_1,\ldots,G_t$, for a positive integer $t$, of a profinite group $G$ such that $G_j\leq N_{G}(G_i)$ for every $j\leq i$; in this setting, we will write
\begin{equation*}
\begin{array}{cccc}
\varphi:& G_1\times \cdots \times G_t &\longrightarrow & G_t\\
 & (g_1,\ldots,g_t) &\longmapsto & [g_1,\ldots,g_t].
\end{array}
\end{equation*}
Let
\begin{equation}\label{eq: coprimes}    
\mathcal{C}=\{(g_1,...,g_t)\in G_1 \times \cdots \times G_t \mid (|g_i|,|g_{i+1}|)=1\};
\end{equation}
for $S_i\subseteq G_i$, $i\in\{1,\ldots,t\}$, we define the set
$$
\varphi^*(S_1,\ldots,S_t)=\varphi((S_1\times\cdots\times S_t)\cap \mathcal{C}).
$$
The subgroups $G_1,\ldots,G_t$ for which the definitions of $\varphi$ and $\varphi^*$ apply will be clear from the context.

For $i\in \{1,\ldots,t\}$, let $X_i,Y_i\subseteq G_i$.
Similar to \cite{dks20}, for $J\subseteq \{1,\ldots,t\}$ we can define the set
$$
\varphi_J (X_i;Y_i)=\varphi(Z_1,\ldots,Z_t) \quad \text{ with } \quad
Z_i=\left\{
\begin{array}{ll}
    X_i & \text{if }\ i\in J,\\
    Y_i & \text{if }\ i\not\in J.
    \end{array}
\right.
$$
Notice that in order for $\varphi_J$ to be well-defined, we just need the subsets $X_i$ where $i\in J$ and the subsets $Y_i$ where $i\not\in J$.
In a similar way, define the set
$$
\varphi_J^* (X_i;Y_i)=\varphi((Z_1\times \cdots \times Z_t)\cap \mathcal{C}).
$$
If $J=\{1,\ldots,t\}$, by abuse of notation, we will just write $\varphi_J (X_i;Y_i)=\varphi(X_i)$ and $\varphi_J^* (X_i;Y_i)=\varphi^* (X_i)$.

\begin{remark}\label{rem: normal in Gj}
    Notice that whenever we have subgroups $G_1, \ldots, G_\ell$ of a profinite group $G$ with $G_j\le N_G(G_i)$ for every $j\le i$, and we take an open subgroup $U\trianglelefteq_o G_\ell$, there exists an open normal subgroup $V\trianglelefteq_o G$ such that $V\cap G_\ell\le U$.
    This implies that $V\cap G_\ell\trianglelefteq G_1\cdots G_\ell$ and $V\trianglelefteq_o G_\ell$.
\end{remark}

The next two lemmas are useful applications of basic commutator calculus.
The first one follows the ideas of Lemma 2.8 of \cite{HeMo21} while Lemma \ref{lem: split subgroups} is an application of Lemma \ref{lem: split commutators Iker Marta} to coprime commutators.

\begin{lemma}
    \label{lem: split commutators Iker Marta}
    Let $G_1,\ldots,G_t$ be subgroups of a profinite group $G$ such that $G_j\le N_G(G_i)$ for every $j\le i$. 
    For every $i\in\{1,\ldots,t\}$ let $g_i\in G_i$, and for a fixed $\ell\in\{1,\ldots,t\}$, let $g_{\ell}'\in G_{\ell}$.
    Then
    $$
    \varphi_{\{\ell\}}(g'_\ell g_\ell;g_{i})=
    [g_1,\ldots,g_{\ell-1},g_\ell',g_{\ell+1}^{h_{\ell}},\ldots,g_t^{h_{t-1}}]^{h_{t}}
    \varphi(g_i),
    $$
    where $h_i\in G_{\ell}\cdots G_{i}$ for $i\in\{\ell,\ldots,t\}$. In particular $g_{i+1}^{h_i}\in G_{i+1}$ for all $i\in \{\ell,\ldots, t-1\}$
    
\end{lemma}
\begin{proof}
    Assume first that $\ell\neq 1$, and proceed by induction on $t-\ell$.
    If $t-\ell=0$, then
    $$
    [g_1,\ldots,g_{t-1},g'_t g_t]
    =
    [g_1,\ldots,g'_t]^{{g_t}[g_1,\ldots,g_t]^{-1}}[g_1,\ldots,g_t],
    $$
    and the result follows.
    Assume $t-\ell>0$, and we write, for the sake of brevity, $y=[g_1,\ldots, g_{\ell-1}]$.
    By induction, we have
    $$
    [y,g'_\ell g_\ell,g_{\ell+1},\ldots,g_t]
    =
    [[y,g_\ell',g_{\ell+1}^{h_{\ell}},\ldots,g_{t-1}^{h_{t-2}}]^{h_{t-1}}[g_1,\ldots,g_{t-1}],g_t]
    $$
    with $h_i\in G_{\ell}\cdots G_i$ for $i\in\{\ell,\ldots,t-1\}$.
    Now,
    \begin{align*}
    [[y,g_\ell',g_{\ell+1}^{h_{\ell}},\ldots,g_{t-1}^{h_{t-2}}]&^{h_{t-1}}[g_1,\ldots,g_{t-1}],g_t]\\
    &\ =
    [[y,g_\ell',g_{\ell+1}^{h_{\ell}},\ldots,g_{t-1}^{h_{t-2}}]^{h_{t-1}},g_t]^{[g_1,\ldots,g_{t-1}]}[g_1,\ldots,g_t]\\
    &\ =
    [y,g_\ell',g_{\ell+1}^{h_{\ell}},\ldots,g_{t-1}^{h_{t-2}},g_t^{(h_{t-1})^{-1}}]^{h_{t-1}[g_1,\ldots,g_{t-1}]}[g_1,\ldots,g_t],
    \end{align*}
    and the lemma follows.
    If $\ell=1$, a similar argument applies.
\end{proof}

\begin{lemma}
\label{lem: split subgroups}
    Let $G_1,\ldots,G_t$ be subgroups of a profinite group $G$ such that $G_j\le N_G(G_i)$ for every $j\le i$.
    Let $\ell\in\{1,\ldots,t\}$ and $H_1,H_2\subseteq G_{\ell}$ be such that $\pi(h_1),\pi(h_2)\subseteq\pi(h_1h_2)$ for every $h_1\in H_1$, $h_2\in H_2$.
    Let $X_i\subseteq G_i$ for $i\in\{1,\ldots,\ell-1\}$, and for $i\in\{\ell+1,\ldots,t\}$ denote $X_i=G_i$.
    Then:
    \begin{enumerate}
        \item  If $\varphi^*_{\{\ell\}}(H_1;X_i)=\varphi^*_{\{\ell\}}(H_2;X_i)=1$, then $\varphi^*_{\{\ell\}}(H_1H_2;X_i)=1$.
        \item If $\varphi^*_{\{\ell\}}(H_j;X_i)=\varnothing$ for some $j\in\{1,2\}$, then $\varphi^*_{\{\ell\}}(H_1H_2;X_i)=\varnothing$.
    \end{enumerate}
\end{lemma}
\begin{proof}
    Since $\pi(h_1),\pi(h_2)\subseteq\pi(h_1h_2)$ for every $h_1\in H_1$, $h_2\in H_2$, the second statement is straightforward.
    Moreover, if $\varphi_{\{\ell\}}(h_1h_2;g_i)\in \varphi_{\{\ell\}}^*(H_1H_2;X_i)$, then for $j\in \{1,2\}$ we have $\varphi_{\{\ell\}}(h_j;g_i)\in \varphi_{\{\ell\}}^*(H_j;X_i)$.
    The result follows now directly from Lemma \ref{lem: split commutators Iker Marta}.
\end{proof}

In view of the preceding lemma, we now introduce a convenient way to choose coset representatives of normal subgroups. These will play an important role throughout the paper.

\begin{definition}
    \label{def: good representatives}
    Let $G$ be a profinite group and $U\trianglelefteq G$.
    An element $g\in G$ is a \emph{good representative} of the coset $gU$ if $\pi(g),\pi(u)\subseteq\pi(gu)$ for every $u\in U$.
\end{definition}

\begin{remark}
\label{rem: good representatives}
    In Definition \ref{def: good representatives}, if $G$ is pronilpotent, then any element $g\in G$ can be written in an unique way as $g=\prod_{p\in \pi(G)} g_p$ with $g_p$ a $p$-element of $G$.
    Hence, $g$ is a good representative of the coset $gU$ if $g_p=1$ whenever $g_p\in U$ for $p\in \pi(G)$, or equivalently, if $\pi(g)$ is minimal among all representatives of the coset $gU$.
\end{remark}

The following lemma is an application of Proposition \ref{prop: Klopsch} to a special type of coprime commutators.

\begin{lemma}
    \label{lem: with Klopsch is easier}
    Let $G_1,\ldots,G_t$ be pronilpotent subgroups of a profinite group $G$ such that $G_j\leq N_G(G_i)$ for all $j\leq i$, and $|G_{\delta^*_{t-1}}|<2^{\aleph_0}$.
    For every $i\in \{1,\ldots,t\}$, let $S_i$ be a closed subset of $G_i$.
    If $\varphi^*(S_i)\neq \varnothing$, then, there exist elements $x_i\in G_i$ and open subgroups $U_i\trianglelefteq_o G_i$ such that $|\varphi^*(x_iU_i\cap S_i)|=1$.
\end{lemma}
\begin{proof}
    Let
    $$
    \mathcal{C}=\Big\{(x_1,\ldots,x_t)\in S_1\times\cdots\times S_t \Bigm| (|x_i|,|x_{i+1}|)=1 \text{ for all } i=1,\ldots,t\Big\}.
    $$
    As $\varphi(\mathcal{C})=\varphi^*(S_i)$, we have $\mathcal{C}\neq\varnothing$.
    Note that $\mathcal{C}$ is closed in $G_1\times \cdots \times G_t$ by Lemma \ref{lem: everything is closed}, and define the map
    \begin{align*}
        \varphi: \mathcal{C} & \longrightarrow G_t\\
        (x_1,\ldots,x_t) & \longmapsto [x_1,\ldots,x_t].
    \end{align*}

    Fix $(x_1,\ldots,x_t)\in \mathcal{C}$. We first prove by induction on $i$ that $g_i:=[x_1,\ldots,x_i]\in G_{\delta^*_{i-1}}$ for every $i\in\{1,\ldots,t\}$ and that $\pi(g_i)\subseteq \pi(x_i)$.
    If $i=1$ the result is obvious, so assume $i>1$ and that $g_{i-1}$ is a $\delta_{i-2}^*$-value with $\pi(g_{i-1})\subseteq \pi(x_{i-1})$, so in particular $(|g_{i-1}|,|x_{i}|)=1$. If $H$ is the minimal Hall subgroup of the pronilpotent group $G_{i}$ containing $x_{i}$, then $g_{i-1}$ acts as a coprime automorphism of $H$. By Lemma \ref{lem: khukhro}, there exists $y_{i}\in H$ such that
    $$
    [x_{i},g_{i-1}]=[y_{i},g_{i-1},\overset{i-1}{\ldots}, g_{i-1}],
    $$
    and Lemma \ref{lem: long commutator of y_i are derived values} shows that $g_i=[x_{i},g_{i-1}]$ is a $\delta_{i-1}^*$-value, as desired.
    Hence, $|\Imm(\varphi)|<2^{\aleph_0}$, and by Proposition \ref{prop: Klopsch}, it follows that there exist elements $x_i\in G_i$ and open normal subgroups $U_i\trianglelefteq G_i$ such that
    $$
    \mathcal{C}\cap (x_1U_1\times\cdots\times x_tU_t) \neq \varnothing
    $$
    and $|\varphi^*(x_iU_i\cap S_i)|=1$.
\end{proof}

Lemma \ref{lem: with Klopsch is easier} will often provide some cosets of open subgroups of $G$ in which coprime commutators are trivial.
Lemmas \ref{lem: cosets to subgroups} and \ref{lem: coset to subgroups*} below will allow us to relate coprime commutators of these cosets with coprime commutators of the open subgroups themselves.

\begin{lemma}
\label{lem: cosets to subgroups}  
    Let $G_1,\ldots,G_t$ be subgroups of a profinite group $G$ such that $G_j \leq N_G(G_i)$ for every $j\le i$, and for every $i\in\{1,\ldots,t\}$, let $x_i\in G_i$ and $U_i\trianglelefteq G_i$.
    Assume also that $G_j\leq N_G(U_i)$ for every $j\le i$.
    Fix $j\in\{1,\ldots,t\}$ and write $J=\{1,\ldots,j-1\}$, then:
\begin{enumerate}
        \item If $\varphi(x_iU_i)=1$ then $\varphi_J(x_iU_i;U_i)=1$.

        \item If
        $\varphi_J(x_iU_i;U_i)=1$ then $$\varphi_{J\cup\{j\}}(x_iU_i;U_i)=\varphi(x_1U_1,\ldots, x_{j-1}U_{j-1},x_j,U_{j+1},\ldots,U_t).$$
    \end{enumerate}
\end{lemma}

\begin{proof}
    (i) We will proceed by reverse induction on $j\in\{1,\ldots,t+1\}$, where the base case $j=t+1$ translates to $\varphi(x_iU_i)=1$, which is true by hypothesis.
    Let thus $j<t+1$ and assume that $\varphi_{J\cup\{j\}}(x_iU_i;U_i)=1$.

    Let $C_{t}=1$ and for every $i\in \{j+1,\ldots,t-1\}$ define $C_{i}=C_{U_{i}}(U_{i+1}/C_{i+1})$.
    Note that $C_i$ is well-defined, since using that for every $\ell$ the subgroup $U_{\ell}$ is normal in $G_1\cdots G_{\ell}$, one can easily show by induction that $C_\ell\trianglelefteq G_1\cdots G_\ell$.
    
    If $j\ge 2$, let
    $$
    Y=\{[x_1u_1,\ldots, x_{j-1}u_{j-1}]\mid u_i\in U_i,\,i=1,\ldots,j-1\}.
    $$
    Then, we can rewrite $\varphi_{J\cup\{j\}}(x_iU_i;U_i)=1$ as
    $$
    [Y, x_{j}U_{j}] \subseteq C_{G_{j}}(U_{j+1}/C_{j+1}).
    $$
    For every $i\in \{1,\ldots,j\}$, fix $u_i\in U_i$ and shorten $y=[x_1u_1,\ldots , x_{j-1}u_{j-1}]$.
    Then we have $[y,x_ju_j]=[y,u_j][y,x_j]^{u_j}$, and since $C_{G_j}(U_{j+1}/C_{j+1})$ is a normal subgroup of $G_j$ containing $[y,x_ju_j]$ and $[y,x_j]$, it follows that $[y,u_j]\in C_{G_j}(U_{j+1}/C_{j+1})$.
    This shows that $\varphi(x_1U_1,\ldots,x_{j-1}U_{j-1},U_{j}, U_{j+1}, \ldots, U_t)=1$, as we wanted.
    
    For the case $j=1$, note that both $x_1$ and $x_1U_1$ lay in $C_{G_1}(U_2/C_2)$, so that $U_1\le C_{G_1}(U_2/C_2)$.
    
    \vspace{5pt}
    
    (ii) For every $i\in\{j+1,\ldots,t\}$ we define $C_i$ as in (i).
    For $i\in\{1,\ldots,t\}$, let $u_i\in U_i$ and shorten $y=[x_1u_1,\ldots,x_{j-1}u_{j-1}]$.
    Then,
    $$
    [y,x_ju_j]=[y,u'x_j]=[y,x_j][y,u']^{x_j}=[y,u']^{x_i[x_j,y]}[y,x_j]
    $$
    for some $u'\in U_j$, and note that $z:=[y,u']^{x_j[x_j,y]}\in C_{G_j}(U_{j+1}/C_{j+1})$.
    Then $[z,u'_{j+1},\ldots,u'_{t}]=1$ for every $u'_i\in U_i$, $i\in\{j+1,\ldots,t\}$, so that
    $$
    [y,x_ju_j,u_{j+1},\ldots,u_t]=[z[y,x_j],u_{j+1},\ldots,u_t]=[y,x_j,u_{j+1},\ldots,u_t],
    $$
    where the last equality follows from Lemma \ref{lem: split commutators Iker Marta}.
    The lemma follows.
\end{proof}

\begin{definition}
    \label{def: Nsigma}
    Let $G_1,\ldots,G_t$ be pronilpotent subgroups of a profinite group $G$ such that $G_j\leq N_G(G_i)$ for all $j\leq i$.
    Let $\sigma$ be a finite set of primes.
    We define the normal subgroup
    $$
    N_{\sigma}=\langle \varphi^*_{\{j\}}(H_i;G_i)\mid j\text{ is such that } |\pi(G_j)|=\infty\rangle^G,
    $$
    where $H_i$ is the Hall $\sigma$-subgroup of $G_i$ for every $i$.
    If $|\pi(G_i)|<\infty$ for all $i$, then $N_\sigma=\langle \varnothing \rangle^G =1$ for every $\sigma$.
\end{definition}

The subgroups $G_1,\ldots,G_t$ of $G$ for which the definition of $N_{\sigma}$ applies will be clear from the context. 
Notice that for any finite sets of primes $\sigma_1$ and $\sigma_2$ such that $\sigma_1\subseteq\sigma_2$ we have
    \begin{equation}
    \label{eq: N tau}
    N_{\sigma_1}\le N_{\sigma_2}.
    \end{equation}

\begin{lemma}
\label{lem: coset to subgroups*}
    Let $G_1,\ldots,G_t$ be pronilpotent subgroups of a profinite group $G$ such that $G_j\leq N_G(G_i)$ for all $j\leq i$.
    Fix $\ell\in\{1,\ldots,t\}$.
    For $i\in\{1,\ldots,\ell-1\}$, let $X_i\subseteq G_i$, and for $i\in\{\ell,\dots,t\}$ let $x_i\in G_i$ and $U_i\trianglelefteq_o G_i$ be such that $G_j\le N_G(U_i)$ for $j\le i$.
    Suppose that $(|x_{\ell}|,|x_{\ell-1}|)=1$ for every $x_{\ell-1}\in X_{\ell-1}$.
    Then:
    \begin{enumerate}
        \item If $(|x_{\ell}|,|U_{\ell+1}|)=1$ and $\varphi^*(X_1,\ldots, X_{\ell-1}, x_{\ell}U_{\ell},U_{\ell+1},\ldots,U_t)=1$, then $\varphi^*(X_1,\ldots, X_{\ell-1}, U_{\ell},\ldots,U_t)=1$.
    
        \item For $i\in\{\ell,\ldots,t\}$, suppose that either $|\pi(G_i)|=\infty$, in which case we write $Y_i=G_i$, or $|\pi(G_i)|=1$, in which case we write $Y_i=U_i$.
        If $\varphi^*(X_1,\ldots, X_{\ell-1}, x_{\ell}U_{\ell},\ldots,x_tU_t)=1,$
        then there exists a finite set of primes $\sigma$ such that
        $\varphi^*(X_1,\ldots, X_{\ell-1}, Y_{\ell},\ldots,Y_t)\subseteq N_{\sigma}$ (cf. Definition \ref{def: Nsigma}).
        Moreover, if the $x_i$ are good representatives of the cosets $x_iU_i$ (cf. Definition \ref{def: good representatives}) such that $(|x_i|,|x_{i+1}|)=1$ for all $i\in\{\ell,\ldots,t-1\}$, then if $\varphi^*(X_1,\ldots, X_{\ell-1}, x_{\ell}U_{\ell},\ldots,x_tU_t)=\varnothing$, also $\varphi^*(X_1,\ldots, X_{\ell-1}, U_{\ell},\ldots,U_t)=\varnothing$.
    \end{enumerate}
\end{lemma}
\begin{proof} 
(i) Since $\varphi^*(X_1,\ldots, X_{\ell-1}, x_{\ell}U_{\ell},U_{\ell+1},\ldots,U_t)\neq\varnothing$, there are $x_1,\ldots,x_{\ell-1}$ such that $x_i\in X_i$ and
    \begin{equation}
        \label{eq: coprime xi}
        (|x_j|,|x_{j+1}|)=1
    \end{equation}
    for all $j\in\{1,\ldots,\ell-2\}$.
    For $i\in\{\ell,\ldots,t\}$, let $H_i$ be a Hall $\pi_i$-subgroup of $U_i$, where  we choose $\pi_i$ such that 
    \begin{equation}
        \label{eq: coprime Hi}
       (|x_{\ell-1}|,|H_\ell|) = (|H_j|,|H_{j+1}|)=1
    \end{equation}
    for all $j\in\{\ell,\ldots,t\}$.
    Since $G_{\ell}$ is pronilpotent, we have $\pi(x_{\ell}h)\subseteq\pi(x_{\ell})\cup\pi(h)$ for all $h\in H_{\ell}$, and hence, as $(|x_{\ell}|,|U_{\ell+1}|)=1$, the set $\varphi(x_1,\ldots,x_{\ell-1},x_{\ell}H_{\ell},H_{\ell+1},\ldots,H_t)$ is contained in $\varphi^*(X_1,\ldots, X_{\ell-1}, x_{\ell}U_{\ell},U_{\ell+1},\ldots,U_t)$, and so it is equal to the trivial subgroup.
    Lemma \ref{lem: cosets to subgroups} now gives $\varphi(x_1,\ldots,x_{\ell-1},H_{\ell},\ldots,H_t)=1$.
    Since this holds for all $x_1,\ldots,x_{\ell-i}$ and for all Hall subgroups of $U_i,i\ge\ell$ satisfying (\ref{eq: coprime xi}) and (\ref{eq: coprime Hi}) respectively, the claim follows.
    
    \vspace{5pt}
    
    (ii) Write $L=\{\ell,\ldots,t\}$, and for $i\in L$, define
    $$
    \sigma_i=\left\{\begin{array}{ll}
    \pi(G_{i}/U_i)  & \text{if } |\pi(G_i)|=\infty,\\
    \pi(G_i) & \text{if } |\pi(G_i)|=1.
    \end{array}\right.
    $$
    Let $\sigma=\sigma_{\ell}\cup\cdots\cup\sigma_{t}$, and assume that the $x_i$ are all good representatives, and in particular that they are all $\sigma$-elements (cf. Remark \ref{rem: good representatives}).
    Furthermore, since $\varphi^*_L(x_iU_i;X_i)\neq\varnothing,$ we may also assume that $(|x_i|,|x_{i+1}|)=1$ for all $i\in\{\ell,\ldots,t-1\}$.
    For $i\in L$ with $|\pi(G_i)|=\infty$, let $V_i$ be the Hall $\sigma'$-subgroups of $G_i$, and for $i\in L$ with $|\pi(G)|=1$, set $V_i=U_i$ (notice that $V_i\le U_i$ if $|\pi(G_i)|=\infty$).
    We can apply (i) $t-\ell +1$ times to obtain $\varphi^*_{L}(V_i;X_i)=1$, where in each step the condition $(|x_{j}|,|V_{j+1}|)=1$ is guaranteed either by the hypothesis that $|\pi(G_j)|=1$ or $\pi(V_j)\subseteq \sigma'$.
    Now, from Lemma \ref{lem: split subgroups}(i), we obtain $\varphi^*_L(Y_i;X_i)\subseteq N_{\sigma}$, as desired. 
    
    Regarding the last claim, if $\varphi^*(X_1,\ldots, X_{\ell-1}, x_{\ell}U_{\ell},\ldots,x_tU_t)=\varnothing$ then we also have $\varphi^*(X_1,\ldots, X_{\ell-1}, x_{\ell},\ldots,x_t)=\varnothing$, so there exists an index $j\in\{1,\ldots,l-2\}$ such that $(|x_j|,|x_{j+1}|)\neq 1$ for all $x_j\in X_j$, $x_{j+1}\in X_{j+1}$, and the lemma follows.
\end{proof}

The following lemma is the focal point of the proof of Proposition \ref{prop: poly-pronilpotent for delta}, as it will allow us to funnel some values of certain coprime commutators into an accurately chosen subgroup.

\begin{lemma}
    \label{lem: going up}
    Let $G_1,\ldots,G_t$ be pronilpotent subgroups of a profinite group $G$ such that $G_j\le N_G(G_i)$ for all $j\le i$, and $|G_{\delta_{t-1}^*}|<2^{\aleph_0}$.
    Then, there exist a finite set $W\subseteq \varphi^*(G_i)$ and a finite set $\sigma$ of primes such that $\varphi^*(G_i)\subseteq N_\sigma \langle W \rangle ^G$.
\end{lemma}
\begin{proof}
    Let
    $$
    \mathcal{I}=\{i\in \{1,\ldots,t\} \mid |\pi(G_i)|=\infty\}.
    $$
    It suffices to prove the theorem in the case when $|\pi(G_i)|=1$ for all $G_i$ with $i\notin \mathcal{I}$. The general case, where each $G_i$, $i\notin \mathcal{I}$, is the product of its Sylow subgroups follows by applying Lemma \ref{lem: split subgroups}.
    
    For $i\notin \mathcal{I}$, let $p_i$ be a prime such that $\pi(G_i)=\{p_i\}$.
    Then we have
    $$
    \varphi^*(G_i)=\varphi^*(G_1,\ldots,G_{i-2}, H_{i-1},G_i,H_{i+1},G_{i+2},\ldots,G_t),
    $$
    where $H_{i-1}$ and $H_{i+1}$ are the Hall $p_i'$-subgroups of $G_{i-1}$ and $G_{i+1}$, respectively.
    We can therefore assume, again by Lemma \ref{lem: split subgroups}(i), that for all $i\notin \mathcal{I}$ we have
    \begin{equation}
        \label{eq: Gi coprime with adjacents}
    \pi(G_i)\not\subseteq \pi(G_{i-1})\cup \pi(G_{i+1}).
    \end{equation}
    
    We claim that that for every $J\subseteq \{1,\ldots,t\}\smallsetminus\mathcal{I}$ there exist a finite set $W_J\subseteq \varphi^*(G_i)$, a finite set of primes $\sigma(J)$ and subgroups $U_i\trianglelefteq_o G_i$ with $i\notin \mathcal{I}\cup J$ such that $\varphi^*_{\mathcal{I}\cup J}(G_i;U_i)\subseteq N_{\sigma(J)}\langle W_J\rangle^G$.
    
    We proceed by induction on $|J|$.
    Assume first $J=\varnothing$.
    By Lemma \ref{lem: with Klopsch is easier}, for every $i\in \{1,\ldots,t\}$ there exist elements $x_i\in G_i$ and subgroups $U_i\trianglelefteq_o G_i$ such that $\varphi^*(x_iU_i)=\{w_\varnothing\}$ for a suitable $w_\varnothing \in G$.
    Moreover, by Remark \ref{rem: normal in Gj}, we may assume that $G_j\le N_G(U_i)$ for every $j\le i$.
    Hence, Lemma \ref{lem: coset to subgroups*}(ii) produces a finite set $\sigma(\varnothing)$ of primes such that $\varphi^*_{\mathcal{I}}(G_i;U_i)\subseteq N_{\sigma(\varnothing)}\langle w_\varnothing\rangle ^G$, so the claim follows for $|J|=0$.
    
    Assume now that $|J|\ge 1$ and that for every $J^{-}\subsetneq J$ there exist a finite set $W_{J^{-}}\subseteq \varphi^*(G_i)$, a finite set of primes $\rho(J^-)$ and subgroups $U^{J^{-}}_i\trianglelefteq_o G_i$, $i\notin\mathcal{I}$, such that $\varphi^*_{\mathcal{I}\cup J^{-}}(G_i;U^{J^{-}}_i)\subseteq N_{\rho(J^-)}\langle W_{J^{-}}\rangle ^G$ (here we are taking $U_i^{J^{-}}=G_i$ if $i\in J^{-}$).
    Let $W_J=\bigcup_{J^-} W_{J^{-}}$, $\rho=\bigcup_{J^{-}}\rho(J^-)$ and $V_i=\bigcap_{J^{-}} U^{J^{-}}_i$, so that, by (\ref{eq: N tau}), we have $\varphi^*_{\mathcal{I}\cup J^{-}}(G_i;V_i)\subseteq N_{\rho}\langle W_J \rangle ^G$ for every $J^{-}\subsetneq J$.
    Furthermore, by factoring out $N_{\rho}\langle W_J \rangle ^G$, we may assume that
    \begin{equation}
    \label{eq: choice of Vi}
    \varphi^*_{\mathcal{I}\cup J^{-}}(G_i;V_i)=1
    \end{equation}
    for every $J^-\subsetneq J$.
    Moreover, taking into account Remark \ref{rem: normal in Gj} we may further assume that $V_i$ is invariant under the conjugacy action of $G_j$ for every $j\le i$.

    Write $J=\{j_1,\ldots,j_n\}$ with $j_1<\cdots<j_n$, and for every $i\in J$, fix a set $S_i$ of coset representatives for $V_i$ in $G_i$ containing the identity.
    Write
    $$
    S_{j_1}\times\cdots\times S_{j_n}=\{\textbf{s}_1,\ldots,\textbf{s}_m\}
    $$
    with $\textbf{s}_{\ell}=(s_{\ell,j_1},\ldots,s_{\ell,j_n})$ for $\ell\in\{1,\ldots,m\}$.
    Denote $V_i=G_i$ for $i\in \mathcal{I}$. Since $1\in S_i$ for every $i$, we have $\varphi^*_{J}(S_i;V_i)\neq \varnothing$, so applying Lemma \ref{lem: with Klopsch is easier} we obtain elements $x_i\in V_i$ and subgroups $U_i\trianglelefteq_o V_i$ such that $\varphi^*_{J}(S_i; x_iU_i)$ takes a single value.
    Actually, since $1=\varphi^*_J(1;x_iU_i)\subseteq \varphi^*_{J}(S_i; x_iU_i)$, we have $\varphi^*_{J}(S_i; x_iU_i)=1$.
    Also, we may assume the $x_i$ to be good representatives of the $U_i$ (cf. Definition \ref{def: good representatives}) and therefore, if $J$ does not contain neither $i$ nor $i+1$, then $(|x_i|,|x_{i+1}|)=1$.
    Thus, for every $\ell\in\{1,\ldots,m\}$, we either have
    \begin{equation}
    \label{eq: recursion phi*}
    \varphi^*_{J}(s_{\ell,i}; x_iU_i)= \varnothing\quad\text{ or }\quad\varphi^*_{J}(s_{\ell,i}; x_iU_i)=1.
    \end{equation}
    Again, by Remark \ref{rem: normal in Gj} we may further assume that $U_i$ is invariant under the conjugacy action of $G_j$ for every $j\le i$.

    Let $J_0=\varnothing$, and for $r\in \{1,\ldots,n\}$, let $J_r= \{j_1,\ldots,j_r\}$.
    We also write $j_0=0$ for convenience.
    We will show that for every $r\in\{0,\ldots,n\}$, there exists a finite set of primes $\sigma(r)$ such that $\varphi^*_{J_{r}}(s_{\ell,i};Y_i^{(r)})\subseteq N_{\sigma(r)}$ for every $\ell\in\{1,\ldots,m\}$, where
    \begin{equation*}
    Y_i^{(r)}=\left\{ 
    \begin{array}{ll}
    G_i & \text{if }\ i \ge j_r,\, i\in \mathcal{I}\cup J,\\
    U_i & \text{if }\ i > j_r,\, i\not\in \mathcal{I}\cup J,\\
    x_iU_i & \text{if }\ i< j_r.
    \end{array}\right.
    \end{equation*}
    
    We argue by reverse induction on $r\in\{0,\ldots,n\}$; assume first $r=n$.
    Since $j_r\not\in\mathcal{I}$, we deduce from (\ref{eq: Gi coprime with adjacents}) that $(|s_{\ell,j_r}|,|G_{j_r-1}|)=(|s_{\ell,j_r}|,|x_{j_r+1}|)=1$.
    Thus, for all $\ell\in \{1,\ldots,m\}$, we obtain from (\ref{eq: recursion phi*}) and Lemma \ref{lem: coset to subgroups*}(ii) a finite set of primes $\sigma(r,\ell)$ such that $\varphi^*_{J_{r}}(s_{\ell,i};Y_i^{(r)})\subseteq N_{\sigma(r,\ell)}$. Defining $\sigma(r)=\bigcup_{\ell=1}^m\sigma(r,\ell)$, we obtain $\varphi^*_{J_{r}}(s_{\ell,i};Y_i^{(r)})\subseteq N_{\sigma(r)}$ for every $\ell\in\{1,\ldots,m\}$.
    
    Hence, we assume $r\le n-1$.
    By induction, we know that there exists a finite set of primes $\sigma(r+1)$ such that
    \begin{equation}
        \label{eq: inductive step r+1}
        \varphi^*_{J_{r+1}}(s_{\ell,i};Y_i^{(r+1)})\subseteq N_{\sigma(r+1)}
    \end{equation}
    for every $\ell\in\{1,\ldots,m\}$.
    We will first show that $\varphi^*_{J_{r}}(s_{\ell,i};Y_i^{(r+1)})\subseteq N_{\sigma(r+1)}$.
    Note that $Y_i^{(r+1)}\le V_i$ for every $i\not\in\mathcal{I}\cup J$ and that $U_{j_{r+1}}\le V_{j_{r+1}}$, so (\ref{eq: choice of Vi}) yields
    \begin{equation}
        \label{eq: Y tilde}
        \varphi^*_{J_r}(s_{\ell,i};\widetilde{Y}_i)\subseteq N_{\sigma(r+1)},
    \end{equation}
    where $\widetilde{Y}_i=Y_i^{(r+1)}$ if $i\neq j_{r+1}$ and $\widetilde{Y}_{j_{r+1}}=U_{j_{r+1}}$.
    As we chose the sets of representatives $S_{j}$ in such a way that the identity is contained in them, for every $\ell\in\{1,\ldots,m\}$, either $s_{\ell,j_{r+1}}$ is trivial or $|\pi(s_{\ell,j_{r+1}})|=1$, so in particular $s_{\ell,j_{r+1}}$ is a good representative.
    Thus, by (\ref{eq: inductive step r+1}) and (\ref{eq: Y tilde}), we deduce from Lemma \ref{lem: split subgroups} that $\varphi^*_{J_r}(s_{\ell,i};\overline{Y}_i)\subseteq N_{\sigma(r+1)}$, where $\overline{Y}_i=Y_i^{(r+1)}$ if $i\neq j_{r+1}$ and $\overline{Y}_{j_{r+1}}=s_{\ell,j_{r+1}}U_{j_{r+1}}$.
    Since this holds for every $\ell\in\{1,\ldots,m\}$, and since $G_{j_{r+1}}=\bigcup_{s\in S_{j_{r+1}}} sU_{j_{r+1}}$, we obtain
    $\varphi^*_{J_r}(s_{\ell,i};Y_i^{(r+1)})\subseteq N_{\sigma(r+1)}$, as we wanted.

    Now using (\ref{eq: Gi coprime with adjacents}) and Lemma \ref{lem: coset to subgroups*}, we conclude exactly as in the case $r=n$ that there exists a finite set $\sigma(r)$ of primes such that $\varphi^*_{J_{r}}(s_{\ell,i};Y_i^{(r)})\subseteq N_{\sigma(r)}$ for every $\ell\in\{1,\ldots,m\}$.
    
    This completes the reverse induction on $r$.
    In particular, for $r=0$, it follows that $\varphi^*_J(G_i;U_i)\subseteq N_{\sigma(0)}$, so this, in turn, concludes the inductive step on $|J|$, and the claim is proved.
    
    Finally, taking $J$ in such a way that $\mathcal{I}\cup J=\{1,\ldots,t\}$, we obtain a finite set of primes $\sigma(J)$ and a finite set $W\subseteq \varphi^*(G_i)$ such that $\varphi^*(G_1,\ldots,G_t) \subseteq N_{\sigma(J)}\langle W\rangle^G$, as desired. 
\end{proof}

Recall that if $G$ is a profinite group of Fitting height $k+1$, then there exist pronilpotent subgroups $T_0,\ldots,T_k$ of $G$ such that $T_j\le N_G(T_i)$ for every $j\le i$.
Moreover, $G=T_0\cdots T_k$, $T_k=\delta_k^*(G)$ and $T_i=[T_0,\ldots,T_i]\pmod{\delta_{i+1}^*(G)}$ for every $i\in\{0,\ldots,k\}$ (see the beginning of Section 3.2).

We remark that $\varphi$ and $\varphi^*$ were defined with variables $\{x_i\mid i=1,\ldots,t\}$ for a generic positive integer $t$.
Since we now want to apply the previous results to the subgroups $T_0,\ldots, T_k$, we will set $t=k+1$ and we will write $\varphi(T_{i-1})$ for $[T_0,\ldots,T_k]$ and $\varphi^*(T_{i-1})$ for $\varphi((T_0\times\cdots\times T_k)\cap \mathcal{C})$, where $\mathcal{C}$ is defined as in (\ref{eq: coprimes}).

\begin{lemma}
\label{lem: values contained in finite normal}
    Let $G=T_0\cdots T_k$ be as above with $T_k=\delta_k^*(G)$ abelian, and assume $|G_{\delta_k^*}|<2^{\aleph_0}$.
    Let $g\in\varphi^*(T_{i-1})$.
    Then, there exists a finite normal subgroup $N\trianglelefteq G$ such that $g\in N$.
\end{lemma}
\begin{proof}
    Write $g=[x_0,\ldots,x_k]$, where $x_j\in T_j$ for all $j$ and $(|x_\ell|,|x_{\ell+1}|)=1$ for all $\ell\in\{0,\ldots,k-1\}$.
    In the same way as in the proof of Lemma \ref{lem: with Klopsch is easier}, we see that $[x_0,\ldots, x_j]$ is a $\delta_j^*$-value for every $j\in \{0,\ldots,k\}$.
 
    In particular $x:=[x_0,\ldots,x_{k-1}]$ is a $\delta_{k-1}^*$-value, and let $H$ be the minimal Hall subgroup of $\delta_k^*(G)$ containing $x_k$, so that $(|x|,|H|)=1$.
    Since, again, $[x,h]$ is a $\delta_k^*$-value for every $h\in H$, the set $K:=\{[x,h]\mid h\in H\}$ has less than $2^{\aleph_0}$ values, and, since $H$ is abelian and normal in $G$, it follows that $K$ is actually a subgroup of $G$.
    In particular, $K$ is finite, so every element of $K$ has finite order.
    Thus, we deduce from Lemma \ref{lem: finite conjugacy class} that the set $S:=\bigcup\{k^G\mid k\in K\}$ is finite, and therefore $N=\langle S\rangle$ is finite.
\end{proof}

We are now ready to prove the strong conciseness of $\delta_k^*$ in prosoluble groups of Fitting height $k+1$. 

\begin{proposition} \label{prop: poly-pronilpotent for delta}
    Let $G$ be a prosoluble group of Fitting height $k+1$.
    Assume that $|G_{\delta^*_k}|< 2^{\aleph_0}$.
    Then $\delta_k^*(G)$ is finite.
\end{proposition}
\begin{proof}
    In view of Lemma \ref{lem: reduce to abelian case}, we may assume that $\delta_k^*(G)$ is abelian.
    Thus, we can take $T_0,\ldots,T_k\le G$ as in Lemma \ref{lem: values contained in finite normal}, so that $G=T_0\cdots T_k$ with $T_k=\delta_k^*(G)$ abelian.
   
    We claim that for every family of subgroups $G_{i-1}\le T_{i-1}$ with $i\in\{1,\ldots,k+1\}$ such that $G_j\le N_G(G_i)$ for $j\le i$, we have $|\varphi^*(G_{i-1})|<\infty$.
    We argue by induction on $|\mathcal{I}|$, where
    $$
    \mathcal{I}=\{i\in \{1,\ldots,k+1\} \mid |\pi(G_{i-1})|=\infty\}.
    $$
    If $|\mathcal{I}|=0$, then Lemma \ref{lem: going up} gives the result since for every finite set $W\subseteq \varphi^*(G_{i-1})$, the normal subgroup $\langle W\rangle^G$ is finite by Lemma \ref{lem: values contained in finite normal}, and since, by definition, $N_\sigma=1$ for every finite set of primes $\sigma$.
    Suppose thus $|\mathcal{I}|\ge 1$.
    Then, Lemma \ref{lem: going up} produces a finite set of primes $\sigma$ and a finite set $W\subseteq \varphi^*(G_{i-1})$ such that $\varphi^*(G_{i-1}) \subseteq N_{\sigma}\langle W \rangle^G$.
    Observe that by induction, for every $j\in\mathcal{I}$, we have $|\varphi^*_{\{j\}}(H_{i-1};G_{i-1})|<\infty$, where $H_{i-1}$ is the Hall $\sigma$-subgroup of $G_{i-1}$, and therefore $N_{\sigma}$ is finite by Lemma \ref{lem: values contained in finite normal}.
    Again by Lemma \ref{lem: values contained in finite normal}, $\langle W \rangle^G$ is also finite, and the claim follows.
    
    In particular, we have shown that $|\varphi^*(T_{i-1})|<\infty$.
    Let $U_0=T_0$; for $j\geq 1$ we construct inductively a series of subgroups $U_j\leq T_j$ in the following way.
    Let $H_{j-1}(p')$ and $Q_{j}(p)$ be, respectively, the Hall $p'$-subgroup of $U_{j-1}$ and the Sylow $p$-subgroup of $T_{j}$, and define
    $$
    U_{j}= \prod_{p\in\pi(G)}[H_{j-1}(p'),Q_{j}(p)].
    $$
    Denote by $P_j(p)=[H_{j-1}(p'),Q_{j}(p)]$ the Sylow $p$-subgroup of $U_j$.
    We claim that $Q_j(p)\equiv P_j(p)\pmod{\delta_{j+1}^*(G)}$ for every $p\in \pi(G)$ and every $j\in\{0,\ldots,k\}$, and in particular that $T_j\equiv U_j\pmod{\delta_{j+1}^*(G)}$.
    The case $j=0$ follows trivially, so let $j\ge 1$ and assume by induction that the congruences hold for $j-1$.
    Let $K_j(p')$ be the Hall $p'$-subgroup of $T_{j-1}T_j$.
    Since $\gamma_{\infty}(T_{j-1}T_j)=T_j\pmod{\delta_{j+1}^*(G)}$, Lemma \ref{lem: Hall acts on Sylow} yields
    $$
    Q_j(p)\equiv[K_{j-1}(p'),Q_{j}(p)]\equiv[H_{j-1}(p'),Q_{j}(p)]\pmod{\delta_{j+1}^*(G)},
    $$
    where the last congruence holds because $T_{j-1}T_{j}\equiv U_{j-1}T_j\pmod{\delta_{j+1}^*(G)}$ by induction, and hence $K_{j-1}(p')\equiv H_{j-1}(p')H^*_{j}(p') \pmod{\delta_{j+1}^*(G)}$, where $H^*_j(p')$ is the Hall $p'$-subgroup of $T_j$.
    The claim follows by the definition of $P_j(p)$.
    In particular, this shows that $U_k=T_k$.

    Furthermore, as $(|Q_j(p)|,|H_{j-1}(p')|)=1$, we have
    $$
    P_j(p)=[Q_{j}(p),H_{j-1}(p')]=[Q_{j}(p),H_{j-1}(p'),H_{j-1}(p')]=[P_j(p),H_{j-1}(p')]
    $$
    by \cite[Lemma 4.29]{Is08}, and thus, for every $j\in\{1,\ldots,k\}$ and very $p\in\pi(G)$ we obtain
    $$
    P_{j}(p)=\prod_{\substack{q\in\pi(G)\\ q\neq p}}[P_{j-1}(q),P_{j}(p)].
    $$
    Therefore, for every $q_k\in \pi(G)$,
    $$
        P_{k}(q_k)=\prod_{(q_0,\ldots,q_{k-1})\in S_{q_k}} [P_0(q_0),\ldots,P_{k}(q_k)],
    $$
    where
    $$
    S_{q_k}=\{(q_0,\ldots,q_{k-1})\in \pi(G)^{k-1}\mid q_j\neq q_{j+1}\ \text{for every}\ j=0,\ldots,k-1\}.
    $$
    By Lemma \ref{lem: split commutators Iker Marta}, this implies that $P_k(q_k)\le\langle \varphi^*(T_{i-1})\rangle$ for every $q_k\in\pi(G)$, and so
    $$
    \delta_k(G)=T_k=\prod_{p\in\pi(G)} P_k(p)=\langle \varphi^*(T_{i-1})\rangle.
    $$    The proposition follows from Lemma \ref{lem: values contained in finite normal}.
    \end{proof}

\section{The main theorems}

We recall that a \emph{minimal simple group} is a finite non-abelian simple group all of whose proper subgroups are soluble.
These groups have been classified by Thompson in \cite{Tho78}.
\begin{lemma}\label{lem: minimal simple groups}
    In every minimal simple group there exist an involution $e$ and an element $h$ of odd order such that $h^e=h^{-1}$. Moreover, for every positive integer $k$, the element
    $$g_k=[h,e,\overset{k-1}{\ldots},e]$$
    is both a non-trivial $\gamma_k^*$-value and a non-trivial $\delta_{k-1}^*$-value.
\end{lemma}
\begin{proof}
     The first claim follows from Theorem 2.13 of \cite{Is08} and the fact that non-abelian simple groups are of even order.
     Notice that $g_k=h^{(-2)^{k-1}}$, so $g_k\neq 1$ for every positive integer $k$. Clearly $g_i$ and $e$ are coprime for every $i\in \{1,\ldots,k-1\}$, so $g_k$ is a $\gamma_k^*$-value.
     Thus, it suffices to prove that the same is true for $\delta^*_{k-1}$.
     By Proposition 25 of \cite{BaMo19}, every involution of a minimal simple group is a $\delta^*_\ell$-value for every $\ell\in \mathbb{N}$.
     Hence, we can use Lemma \ref{lem: long commutator of y_i are derived values} with $g_1=\cdots=g_{k-1}=e$ and $H=\langle h \rangle$ and conclude the proof.
\end{proof}

We are now ready to prove our main results.
As in \cite{dms21}, we start showing that the Fitting subgroup of any profinite group $G$ with $|G_{\gamma_k^*}|<2^{\aleph_0}$ or $|G_{\delta_k^*}|<2^{\aleph_0}$ is infinite.

\begin{proposition}\label{prop: infinite Fitting}
    Let $G$ be an infinite profinite group and let $w^*=\delta_k^*$ or $w^*=\gamma_k^*$. Suppose that $|G_{w^*}|<2^{\aleph_0}$. Then the Fitting subgroup $F$ of $G$ is infinite.
\end{proposition}
\begin{proof}
    We first show that $F$ is non-trivial. Assume by contradiction that $F=1$. For every $w^*$-value $x$ of $G$, the normal closure $\langle x^G \rangle$ is finite.
    Indeed, by Lemma \ref{lem: finite conjugacy class}, $x^G$ is finite, so in particular $|G: C_G(\langle x^G \rangle)|<\infty$.
    As a consequence, the index $|\langle x^G \rangle:Z(\langle x^G \rangle)|$ is also finite, but $Z(\langle x^G \rangle)$ is contained in $F=1$.
    This implies that $G$ possesses finite minimal normal subgroups, so let $N$ be the product of all of them.
    
    If $N$ is finite, then there exists a normal open subgroup $K\trianglelefteq_o G$ such that $K\cap N=1$. Such a subgroup cannot contain any $w^*$-value since otherwise, repeating the same argument as before, we would obtain a minimal normal subgroup of $G$ contained in $K$, contradicting that $K\cap N=1$.
    If $w^*(K)=1$, then by Theorem \ref{aaaa} $K$ is either pronilpotent (if $w^*=\gamma_k^*$) or prosoluble of Fitting height $k$ (if $w^*=\delta_k^*$), and this contradicts the fact that $F\cap K=1$. This proves that $N$ is an infinite subgroup of $G$.
    
    None of the infinitely many minimal normal subgroups of $G$ is abelian because $F=1$, so each of these minimal normal subgroup contains a section isomorphic to a minimal simple group.
    For each minimal normal subgroup $N_i$, with $i\in I$, choose a section isomorphic to a minimal simple group $S_i$.
    We remark that by the previous discussion $I$ is an infinite set and so the Cartesian product of $S_i$ is a section of $G$. By Lemma \ref{lem: minimal simple groups} in each of these groups $S_i$ there exist an involution $e_{S_i}\in S_i$ and an element $h_{S_i}\in S_i$ of odd order with $h_{S_i}^{e_{S_i}}=h_{S_i}^{-1}$ such that $g_{S_i}:=[h_{S_i},e_{S_i},\overset{k-1}{\ldots},e_{S_i}]$ is a non-trivial $w^*$-value.
    Now, for each subset $J\subseteq I$, the element $c_J=\prod_{j\in J} g_{S_j}$ can be written as
    $$c_J=\Big[\prod_{j\in J} h_{S_j},\prod_{j\in J} e_{S_j},\ldots, \prod_{j\in J} e_{S_j}\Big].$$
    Clearly $\prod_{j\in J} e_{S_j}$ is an involution normalizing the cyclic subgroup generated by the element of odd order $\prod_{j\in J} h_{S_j}$, and hence it is a $w^*$-value (if $w^*=\delta_k^*$ we also need to use Lemma \ref{lem: long commutator of y_i are derived values}).
    However, there exist at least $2^{\aleph_0}$ distinct subsets $J\subseteq I$ that give rise to different $c_J$, against the assumption that $|G_{w^*}|<2^{\aleph_0}$, so $F\neq 1$.
    
    If we assume by contradiction that the Fitting subgroup $F$ is finite, then there would be a subgroup $K\trianglelefteq_o G$ with $K\cap F=1$, so that $K$ has trivial Fitting subgroup.
    By the previous argument, this can happen only if $w^*(K)=1$, so $K$ is either pronilpotent or prosoluble of Fitting height $k$, contradicting that $K\cap F=1$ and proving the proposition.
\end{proof}

\begin{proof}[Proofs of Theorems \ref{thm: gamma strongly concise} and \ref{thm: delta strongly concise}] In view of Theorem \ref{aaaa} it is sufficient
to show that if $|G_{w^*}|<2^{\aleph_0}$, then $G$ is finite-by-pronilpotent in the case $w^*=\gamma_k^*$ or finite-by-(prosoluble of Fitting height $k$) if $w^*=\delta_k^*$.
    We can assume $G$ to be infinite, otherwise the theorem is trivially true.

    We will denote the Fitting subgroup of $G$ by $F$, and for $i\ge 2$, let $F_i$ be the $i$-th Fitting subgroup of $G$.
    By Proposition \ref{prop: infinite Fitting}, $F$ is infinite (and hence the same is true for all $F_k$).
    Let $n=2$ if $w^*=\gamma_k^*$ and $n=k+1$ if $w^*=\delta_k^*$.
    By Propositions \ref{prop: meta-pronilpotent for gamma} and \ref{prop: poly-pronilpotent for delta}, $w^*(F_n)$ is finite.
    Therefore there exits an open normal subgroup $R\trianglelefteq_o F_n$ with $R\cap w^*(F_n)=1$.
    Theorem \ref{aaaa} implies that the Fitting height of $R$ is at most $n-1$, and hence $R$ is contained in $F_{n-1}$.
    However, since $F_n/F_{n-1}$ is the Fitting subgroup of $G/F_{n-1}$, it follows that $G/F_{n-1}$ has finite Fitting subgroup, and by Proposition \ref{prop: infinite Fitting} this can only happen if $G/F_{n-1}$ is finite.

    Thus, we will prove the result by induction on $|G:F_{n-1}(G)|$, with the base case $G=F_{n-1}(G)$ being trivial.
    Assume then that $|G:F_{n-1}(G)|>1$ and suppose first that $G/F_{n-1}$ has a nontrivial proper normal subgroup $N$.
    The inductive hypothesis yields $|w^*(N)|<\infty$, and working in $G/w^*(N)$, we obtain by Theorem \ref{aaaa} that $N/w^*(N)$ is prosoluble of Fitting height at most $n-1$.
    This implies that $N/w^*(N)$ is contained in the $(n-1)$-th Fitting subgroup of $G/w^*(N)$ and by inductive hypothesis $w^*(G/w^*(N))$ must be finite, so $w^*(G)$ is finite too.
    
    We can hence assume that $G/F_{n-1}$ is a simple group.
    Notice that if $G/F_{n-1}$ is abelian, then we can conclude simply by applying Proposition \ref{prop: meta-pronilpotent for gamma} or Proposition \ref{prop: poly-pronilpotent for delta}.
    Thus, the only case left is when $G/F_{n-1}$ is a finite non-abelian simple group.
    By Theorem \ref{aaaa} we have $w^*(G/F_{n-1})=G/F_{n-1}$, so there is a finite set $S$ consisting of $w^*$-values such that $G=\langle S \rangle F_{n-1}$.
    By Lemma \ref{lem: finite conjugacy class} the set $T:=\bigcup\{s^G\mid s\in S\}$ is finite, so the index $|G:C_G(T)|$ is also finite.
    This implies that the center of $\langle T \rangle$ has finite index in $\langle T\rangle$, so by Schur's theorem $\langle T \rangle'$ is finite too. Note that $\langle T \rangle'$ is normal in $G$.
    Factoring out $\langle T \rangle'$, we can assume $G/F_{n-1}$ to be abelian, and we conclude the proof as before.
\end{proof}

\end{document}